\documentclass[11pt]{amsproc}
\usepackage{amsmath,amssymb}
\usepackage[mathcal]{euscript}
\usepackage{amsthm}
\usepackage{amsfonts}
\usepackage{amsmath,amssymb}
\usepackage{indentfirst,latexsym,bm,amsthm,graphicx}
\usepackage{fancyhdr}
\usepackage{mathrsfs}
\usepackage{times}
\usepackage{amsmath,amsfonts,amssymb,amsthm}
\usepackage{epsfig}
\usepackage[all]{xy}
\usepackage{tikz}
\usepackage{color}
\usepackage{enumerate}
\usepackage{verbatim}

\newtheorem{theo}{Theorem}[section]
\newtheorem{lemmm}{Lemma}[section]

\newtheorem{lem}{Lemma}[section]
\newtheorem{prop}{Proposition}
\newtheorem{propp}{Proposition}[section]
\theoremstyle{definition}

\newtheorem{defini}{Definition}[section]
\newtheorem{example}{Example}

\newtheorem{coro}{Corollary}
\theoremstyle{remark}

\allowdisplaybreaks

\numberwithin{equation}{section}

\def \<{\langle}
\def \>{\rangle}

\def \be{\begin{equation}\label}
\def \ee{\end{equation}}
\def \bex{\begin{example}\label}
\def \eex{\end{example}}
\def \bl{\begin{lem}\label}
\def \el{\end{lem}}
\def \bt{\begin{thm}\label}
\def \et{\end{thm}}
\def \bp{\begin{prop}\label}
\def \ep{\end{prop}}
\def \br{\begin{rem}\label}
\def \er{\end{rem}}
\def \bc{\begin{coro}\label}
\def \ec{\end{coro}}
\def \bd{\begin{de}\label}
\def \ed{\end{de}}

\begin{document}
\title[Quantum Supergroup, Scasimir Operators, and Dickson polynomials]
{Quantum Supergroup ${U_{r,s}}\left( {osp(1,2)} \right)$, Scasimir Operators, and Dickson polynomials}

\author[Liu]{Fu Liu}
\address{Department of Mathematics, Shanghai University,
Shanghai 200444, China} \email{1279897762@shu.edu.cn}

\author[Hu]{Naihong Hu}
\address{School of Mathematical Sciences, Shanghai Key Laboratory of PMMP, East China Normal University
Shanghai 200241, China} \email{nhhu@math.ecnu.edu.cn}

\author[Jing]{Naihuan Jing$^\star$}
\address{
Department of Mathematics, North Carolina State University,
   Raleigh, NC 27695, USA \\}
\email{jing@ncsu.edu}
\thanks{$*$ Corresponding author}

\keywords{Two-parameter quantum supergroup, center, Dickson polynomial, Scasimir Operator.}
\begin{abstract}
We  study the center of the two-parameter quantum supergroup ${U_{r,s}}\left( {osp(1,2)} \right)$ using the Dickson polynomial. We show that the Scasimir operator is completely determined by the $q$-deformed Chebychev polynomial, generalizing an earlier work of Arnaudon and Bauer.
\end{abstract}

\maketitle

\section{\textbf{Introduction}}

  Using down-up algebras, Benkart and Withspoon \cite{BW1, BW2} recovered Takeuchi's definition of the two-parameter quantum group of type $A_n$. Since then, a systematic
 study of two-parameter quantum groups has been carried out, for instance \cite{BW3, JL1, JL2, BW4} for type $A_n$. \cite{BGH1, BGH2, HSY} for types $B_n$, $C_n$, $D_n$; \cite{BH, CH, HSQ, HW} for types ${F_4}$, ${E_6}$, ${E_7}$, ${E_8}$, ${G_2}$; \cite{HP1, HP2} for any types, respectively.
Some work has also been done on the structure and representation theory of the two-parameter quantum supergroups \cite{AY, Sh, Zh, CL, HM}.

For $U_q(\mathfrak g)$ with $q$ generic, the center of the quantized enveloping algebra of a semisimple Lie algebra is well-known to be described by the Harish-Chandra homomorphism, that is, $Z\left( {{U_q}\left( \mathfrak{g} \right)} \right) \cong {\left( {U_{ev}^0} \right)^W}$, where ${\left( {U_{ev}^0} \right)^W}$ denotes the span of all quantum Cartan elements ${k_\mu }$ with $\mu$ ``even" weights, i.e., in $2\Lambda$. This naturally produces the quantum Casimir operator analogous to the semisimple case, and here we still call it the Casimir operator, which can be extended to the two-parameter quantum case. As for a more recent work of ${\left( {U_{ev}^0} \right)^W}$ described as a polynomial algebra in several variables or a certain quotient object of it, see \cite{LXZ1, LXZ2}. Recently, the Harish-Chandra type theorem for $U_q(\mathfrak g)$ with $\mathfrak g$ being a simple basic Lie superalgebra (cf. \cite{Kac}) when $q$ is generic has been obtained by Luo-Wang-Ye in \cite{LWY}.

 However, it is known that in the case when $q$ is a root of unity, the Casimir operators in general do not exhaust the whole center of the quantum (super)group. For instance, some powers of generators also belong to the center. All these central elements are however not independent, but satisfy certain polynomial relations \cite{Lu, AB1, AAB, ABF, ACF, AB2}.
 On the other hand, such relations in the center of a quantum supergroup have not been known in general, except for some special examples \cite{AAB,ABF, ACF}, \cite{Ke, AB2}. In \cite{AB2}, the authors used the Chebychev polynomial to give explicit relations between the Casimir operator, Scasimir operator and the powers of generators in the center.
 Corresponding result for the quantum supergroup ${U_{r,s}}\left( {osp(1,2)} \right)$ was discussed in  \cite{Sh}, but the relations between the Casimir operator, the Scasimir operator in terms of generators were absent for root of unity case.
 The aim of this paper is to derive these relations using the Dickson polynomial, which is a certain deformation of the Chebychev polynomial (cf. \cite{LMT}).

 Denote the two-parameter quantum supergroup ${U_{r,s}}\left( {osp(1,2)} \right)$ by $U$. The article is divided into four parts as follows.
 In Sect. 2, we recall the definition of $U$ and describe the Poincar\'e-Birkhoff-Witt (PBW) basis of $U$ as well as the relations among generators. The non-zero-divisor property of $U$ when $r{s^{-1}}$ is a root of unity, and the characterization of the center of $U$ are considered in Sect. 3. Subsequently in Sect. 4, we recall the Dickson polynomial and its generating function. Finally in Sect. 5, the relation formulas between the Casimir operator, the Scasimir operator and the powers of generators in the center are determined by using the Dickson polynomial.

\section{\textbf{Quantum  supergroup ${U_{r,s}}\left( {osp(1,2)} \right)$}}

Let $r,s$ be generic with ${r^2} \ne {s^2}$. We denote by ${\mathscr A} = \mathbb{C}\left[ {{{\sqrt r }^{ \pm 1}},{{\sqrt s }^{ \pm 1}}} \right]$, and by ${\mathscr F} = \mathbb{C}\left( {\sqrt r ,\sqrt s } \right)$ the quotient field of ${\mathscr A}$. Setting $q = {r^{\frac{1}{2}}}{s^{\frac{1}{2}}}$, $q' =  - q$, $\eta  = ({r^{\frac{1}{2}}} + {s^{\frac{1}{2}}})(r - s)$. Also let $z$ be an indeterminate $z$ $ \in \mathbb{C}$.

In the following, we  often study the case when $q$ is a primitive ${\l ^{th}}$-root of unity.  We use ${\l'}$ to denote the order of $q'$ and $L$ to denote the smallest even multiple of ${\l '}$, i.e., $\l$ if ${\l}$ is even and $2\l$ if $\l$ is odd.

The two-parameter quantum supergroup  ${U_{r,s}}\left( {osp(1,2)} \right)$ is the unital associative algebra over ${\mathscr F}$ generated by $e,f,{w^{ \pm 1}},{w^{' \pm 1}}$ and subject to the following relations \cite{Sh}:
\begin{eqnarray}\label{n:tor1}
&&\omega\omega^{'}=\omega^{'}\omega, \quad \omega\omega^{-1}=\omega^{-1}\omega=1=\omega^{'}\omega^{'-1}=\omega^{'-1}\omega^{'},\\ \label{n:tor2}
&&\omega e\omega^{-1}=q e, \quad \omega^{'} e\omega^{'-1}=q^{-1}e,\\\label{n:tor3}
&&\omega f\omega^{-1}=q^{-1}f, \quad\omega^{'} f\omega^{'-1}=q f,\\\label{n:tor4}
&&ef + fe = \frac{{\omega  - {\omega ^{'}}}}{{r - s}}.\label{n:tor5}
\end{eqnarray}

 The ${\mathbb{Z}_2}$-grading $\deg (e) = \deg (f) = 1$, $\deg (\omega ) = \deg ({\omega ^{'}}) = 0$ is compatible with the above relations, hence can be extended to a grading of $U$. Then we have the ${\mathbb{Z}_2}$-graded  decomposition: $U = {U_0} \oplus {U_1}$. The elements in the first (resp. second) summand are called bosonic elements (resp. fermionic elements). Clearly ${U_0}$ is a subalgebra of $U$.

We are interested in the center. The {\it Scasimir operator} $\widetilde{c}$ is defined by
\begin{equation*}
\widetilde{c}={r^{{\textstyle{1 \over 2}}}}\omega  - {s^{{\textstyle{1 \over 2}}}}{\omega ^{'}} - \eta fe.
\end{equation*}
For convenience, we change variables and
set $b = {r^{{\textstyle{1 \over 4}}}}{\omega ^{{\textstyle{1 \over 2}}}},v = {s^{ - {\textstyle{1 \over 4}}}}{\omega ^{' - {\textstyle{1 \over 2}}}}$ and  ${b^{ - 1}}v\widetilde{c} = u - \frac{1}{u}$.
The {\it Casimir operator} $C$ is then given by
\begin{equation*}
C = {b^{ - 2}}{v^2}{\widetilde{c}}^{2} = {(rs)^{ - {\textstyle{1 \over 2}}}}{(\omega {\omega ^{'}})^{ - 1}}{\widetilde{c}}^{2}.
\end{equation*}
The following result is well-known:
\begin{propp} \cite{Le}
The Scasimir operator $\widetilde{c}$ anti-commutes with fermionic elements and commutes with bosonic elements.
\end{propp}

The following result will be useful.
\begin{lem}\label{l4.2}
For any $a,b \in \mathbb{N}$, the following equations hold in $U$:
\begin{eqnarray}\label{10}
 && {f^{m}}e + {( - 1)^{m - 1}}e{f^{m}} = {f^{m - 1}}\left( {\frac{{{q^{' - m}} - 1}}{{{q^{' - 1}} - 1}}\omega  - \frac{{{q^{'m}} - 1}}{{{q^{'}} - 1}}{\omega ^{'}}} \right){(r - s)^{ - 1}}, \\
\label{11}
 &&  {e^m}f + {( - 1)^{m - 1}}f{e^m} = {e^{m - 1}}\left( {\frac{{{q^{'m}} - 1}}{{{q^{'}} - 1}}\omega  - \frac{{{q^{' - m}} - 1}}{{{q^{' - 1}} - 1}}{\omega ^{'}}} \right){(r - s)^{ - 1}}.
\end{eqnarray}
\end{lem}
\begin{proof}
These can be directly verified by the defining relations $(2.1)$---$(2.4)$.
\end{proof}

Note that $\widetilde{c}$ is a square root of $C$ up to a constant. In fact, there is a finer relation among $\widetilde{c}$,
$\omega ,{\omega ^{'}},f,e$.
\begin{propp} For a positive integer $m$, let
\begin{gather*}
\varepsilon (m)=\begin{cases} 1, & m = 0 \ or \ 1 \;(\bmod\; 4);\\ -1, & m = 2 \ or \ 3 \;(\bmod\; 4). \end{cases}
\end{gather*}
Then
\begin{eqnarray*}
\mathop \Pi \limits_{n = 0}^{m - 1} \left( {\widetilde{c} - {q^{'n}}{r^{{\textstyle{1 \over 2}}}}\omega  + {q^{' - n}}{s^{{\textstyle{1 \over 2}}}}{\omega ^{'}}} \right) = \varepsilon (m)( - \eta ){}^{m}{f^{m}}{e^{m}}.
\end{eqnarray*}
\end{propp}
\begin{proof}
When $m=1$, the equation is the definition of $\widetilde{c}$.
Suppose for $m=t\; (t \ge 1)$, this equation holds. Then for $m=t+1$,
\begin{equation*}
\begin{split}
&f\mathop \Pi\limits_{n = 0}^{t - 1} \left( {\widetilde{c} - {q^{'n}}{r^{{\textstyle{1 \over 2}}}}\omega  + {q^{' - n}}{s^{{\textstyle{1 \over 2}}}}{\omega ^{'}}} \right)e( - \eta )\varepsilon (t)\varepsilon (t + 1)\\
&\qquad\qquad= - \eta {\varepsilon ^2}(t)\varepsilon (t + 1)( - \eta ){}^t{f^{t + 1}}{e^{t + 1}} \\
&\qquad\qquad=\varepsilon (t + 1)( - \eta ){}^{t + 1}{f^{t + 1}}{e^{t + 1}}.
\end{split}
\end{equation*}
Apply the commutation relations of $f$ with $\omega, \omega', \widetilde{c}$:  
\begin{gather*}
f\omega  =  - {q^{'}}\omega f, \quad f{\omega ^{'}} =  - {q^{' - 1}}{\omega ^{'}}f,\quad f\widetilde{c} =  - \widetilde{c}f
\end{gather*}
to the left side of the equation, then
\begin{gather*}
LHS  =  ( - 1){}^t\varepsilon (t)\varepsilon (t + 1)\mathop \Pi \limits_{n = 0}^t \left( {\widetilde{c} - {q^{'n}}{r^{{\textstyle{1 \over 2}}}}\omega  + {q^{' - n}}{s^{{\textstyle{1 \over 2}}}}{\omega ^{'}}} \right).
\end{gather*}
By definition of $\varepsilon(m)$, it is easy to verify $( - 1){}^t\varepsilon (t)\varepsilon (t + 1)=1$.
So we have proved the relations among $\widetilde{c}$, $k, k^{-1}, f$, $ e$.
\end{proof}

The following result describes a PBW basis of $U$.
\begin{propp}
The set $\left\{ {{f^a}{e^b}{\omega ^c}{\omega ^{'d}}\left| {a,b \in \mathbb{N},c,d \in \mathbb{Z}} \right.} \right\}$ is a set of linearly independent basis of $U$.
\end{propp}
\begin{proof}
The basis of $U's$ is given by the means of the Ore extension \cite{Kas}. Here we use the representation-theoretic method to characterize the basis of $U$. This proof  is similar to that of the PBW-basis of ${U_q}(s{l_2})$ in \cite{Ja}. Let $A = span\left\{ {{f^a}{e^b}{\omega ^c}{\omega ^{'d}}\left| {a,b \in \mathbb{N},c,d \in \mathbb{Z}} \right.} \right\}$. Firstly, for any element $x$ in U, we  prove that $xA \subset A$. For $x=f$, $f  {f^a}{e^b}{\omega ^c}{\omega ^{'d}} = {f^{a + 1}}{e^b}{\omega ^c}{\omega ^{'d}}$. If $x = \omega ,{\omega ^{'}}$, then, from $(2.2)$ and $(2.3)$, we have
\begin{gather*}
\omega {f^a}{e^b}{\omega ^c}{\omega ^{'}}^d = {q^{b - a}}{f^a}{e^b}{\omega ^{c + 1}}{\omega ^{'}}^d,\\
{\omega ^{'}}{f^a}{e^b}{\omega ^c}{\omega ^{'}}^d = {q^{a - b}}{f^a}{e^b}{\omega ^c}{\omega ^{'}}^{d + 1}.
\end{gather*}
If $x=e$, by Equation $(2.5)$, we have
\begin{eqnarray*}
e{f^a}{e^b}{\omega ^c}{\omega ^{'}}^d &=&{(r - s)^{ - 1}}{f^{a - 1}}{e^b}\left( {{q^b}\frac{{{q^{ - a}} + ( - 1){}^{a - 1}}}{{{q^{ - 1}} + 1}}\omega  - {q^{ - b}}\frac{{{q^{  a}} + ( - 1){}^{a - 1}}}{{q + 1}}{\omega ^{'}}} \right){\omega ^c}{\omega ^{'}}^d\\
&&+\;( - 1){}^a{f^a}{e^{b + 1}}{\omega ^c}{\omega ^{'}}^d.
\end{eqnarray*}
Observe that
\begin{equation*}
{(r - s)^{ - 1}}\left( {{q^b}\frac{{{q^{ - a}} + ( - 1){}^{a - 1}}}{{{q^{ - 1}} + 1}}\omega  - {q^{ - b}}\frac{{{q^{ a}} + ( - 1){}^{a - 1}}}{{q + 1}}{\omega ^{'}}} \right)
\end{equation*}
is a polynomial with respect to $\omega$ and $\omega^{'}$, so $xA \subset A$. This proves that any element in  $U$ can be linearly expressed by  monomials $\left\{ {{f^a}{e^b}{\omega ^c}{\omega ^{'d}}\left| {a,b \in \mathbb{N},c,d \in \mathbb{Z}} \right.} \right\}$. Consider the polynomial ring $\mathbb{C}[X,Y,Z,W]$ with indeterminate elements $X,Y,Z,W$ and its localization
$$B \equiv \mathbb{C}[X,Y,Z,W,{Z^{ - 1}},{W^{ - 1}}],$$
then all monomials ${X^a}{Y^b}{Z^c}{W^d}$ with $a,b \in \mathbb{N},c,d \in \mathbb{Z}$ form a basis of $B$. Define the endomorphisms $E,F,\Omega,\Omega^{'}$ of $B$ by
\begin{gather*}
E({X^a}{Y^b}{Z^c}{W^d})
={(r - s)^{ - 1}}{X^{a - 1}}{Y^b}\left( {{q^b}\frac{{{q^{ - a}} + ( - 1){}^{a - 1}}}{{{q^{ - 1}} + 1}}Z - {q^{ - b}}\frac{{{q^a} + ( - 1){}^{a - 1}}}{{q + 1}}W} \right){Z^c}{W^d}\\
+\; ( - 1){}^a{X^a}{Y^{b + 1}}{Z^c}{W^d},\\
F({X^a}{Y^b}{Z^c}{W^d})={X^{a + 1}}{Y^b}{Z^c}{W^d}, \\
\Omega ({X^a}{Y^b}{Z^c}{W^d})={q^{b - a}}{X^a}{Y^b}{Z^{c + 1}}{W^d},\\
{\Omega ^{'}}({X^a}{Y^b}{Z^c}{W^d})={q^{a - b}}{X^a}{Y^b}{Z^c}{W^{d + 1}}.
\end{gather*}
It is not difficult to see that both $\Omega$ and ${\Omega ^{'}}$ are bijective with
\begin{gather*}
{\Omega ^{ - 1}}({X^a}{Y^b}{Z^c}{W^d}) = {q^{a - b}}{X^a}{Y^b}{Z^{c - 1}}{W^d}, \\
{\Omega ^{' - 1}}({X^a}{Y^b}{Z^c}{W^d}) = {q^{b - a}}{X^a}{Y^b}{Z^c}{W^{d - 1}}.
\end{gather*}
We now check  $(E,F,\Omega ,{\Omega ^{'}})$ satisfy the relations $(2.1)$---$(2.4)$. We only verify $(2.4)$, the other relations
are similarly proved.
\begin{equation*}
\begin{split}
(EF &+ FE)({X^a}{Y^b}{Z^c}{W^d})\\
&={(r - s)^{ - 1}}{X^{a}}{Y^b}\left( {{q^b}\frac{{{q^{ - a - 1}} + ( - 1){}^a}}{{{q^{ - 1}} + 1}}Z - {q^{ - b}}\frac{{{q^{a + 1}} + ( - 1){}^a}}{{q + 1}}W} \right){Z^c}{W^d}\\
&\quad+ ( - 1)^{a + 1}{X^{a + 1}}{Y^{b + 1}}{Z^c}{W^d}\\
&\quad + {(r - s)^{ - 1}}{X^a}{Y^b}\left( {{q^b}\frac{{{q^{ - a}} + ( - 1){}^{a - 1}}}{{{q^{ - 1}} + 1}}Z - {q^{ - b}}\frac{{{q^a} + ( - 1){}^{a - 1}}}{{q + 1}}W} \right){Z^c}{W^d} \\
&\quad+ ( - 1)^a{X^{a + 1}}{Y^{b + 1}}{Z^c}{W^d}\\
&={(r - s)^{ - 1}}{X^a}{Y^b}\left( {{q^{b - a}}Z - {q^{a - b}}W} \right){Z^c}{W^d}\\
&={(r - s)^{ - 1}}(\Omega  - {\Omega ^{'}})({X^a}{Y^b}{Z^c}{W^d}).
\end{split}
\end{equation*}
Therefore, there is a homomorphism : $U \to End(B)$ that takes $e \mapsto E, f \mapsto F, {w^{ \pm 1}} \mapsto {\Omega ^{ \pm 1}}, and \
  {w^{' \pm 1}} \mapsto {\Omega ^{' \pm 1}}$, hence it sends the monomial ${f^a}{e^b}{\omega ^c}{\omega ^{'}}^d \mapsto {F^a}{E^b}{\Omega ^c}{\Omega ^{'}}^d.$  Notice that ${F^a}{E^b}{\Omega ^c}{\Omega ^{'}}^d(1) = {X^a}{Y^b}{Z^c}{W^d}$,  this shows that $\left\{ \left.{F^a}{E^b}{\Omega ^c}{\Omega ^{'d}}\;\right|\; {a,b \in \mathbb{N},c,d \in \mathbb{Z}} \right\}$ is linearly independent, therefore the set $\left\{\left.{f^a}{e^b}{\omega ^c}{\omega ^{'d}}\;\right|\; {a,b \in \mathbb{N},c,d \in \mathbb{Z}} \right\}$ is linearly independent.
\end{proof}

 \section{\textbf{Supercenter}}
\begin{theo}\label{main}
   When $r{s^{ - 1}}$ is not a root of unity, the center of $U$ is $\mathbb{C}[\omega\omega^{'} ,{\widetilde{c}}^{2}]$.
   \end{theo}
\begin{proof}
Consider the commutative subalgebra $T = \mathbb{C}[\widetilde{c},{\omega ^{ \pm 1}},{\omega ^{' \pm 1}}]$. Applying Proposition $2.2$ to the monomial ${f^a}{e^b}{\omega ^c}{\omega ^{'d}}$ $(m=min(a,b))$, we see that

\begin{equation*}
\begin{split}
{f^a}&{e^b}{\omega ^c}{\omega ^{'d}}\\
 &={f^{a - m}}{f^m}{e^m}{e^{b - m}}{\omega ^c}{\omega ^{'d}}\\
 &={f^{a - m}}\Bigl[\varepsilon (m){( - \eta )^{ - m}}\prod\limits_{n = 0}^{m - 1} {\Bigl(c - {q^{'n}}{r^{{\textstyle{1 \over 2}}}}\omega  + {q^{' - n}}{s^{{\textstyle{1 \over 2}}}}{\omega ^{'}}\Bigr)} \Bigr]{e^{b - m}}{\omega ^c}{\omega ^{'d}}\\
 &=\varepsilon (m){( - \eta )^{ - m}}{( - 1)^{(a - m)m}}{q^{(a - b)(c - d)}}\Bigl[\prod\limits_{n = 0}^{m - 1} {\Bigl(c - {q^{'n + a - m}}{r^{{\textstyle{1 \over 2}}}}\omega  + {q^{' - (n + a - m)}}{s^{{\textstyle{1 \over 2}}}}{\omega ^{'}}\Bigr)}\Bigr]\\
 &\quad\times{\omega ^c}{\omega ^{'d}}{f^{a - m}}{e^{b - m}},
\end{split}
\end{equation*}
This implies $U$ is a free $T$-module with basis $1,f,e,{f^2},{e^2}, \ldots$. For any $i \in \mathbb{N}$, $\omega {f^i} = {q^{ - i}}{f^i}\omega $, $\omega {e^i} = {q^i}{e^i}\omega$. Since $q$ is not a root of unity, so the centralizer of $\omega$ in $U$ is $T$. Similarly, the centralizer of ${{\omega ^{'}}}$ in $U$ is also $T$. Next we determine the centralizer of $f$ and $e$ in $T$.

\

Let $P(\omega ,{\omega ^{'}},\widetilde{c}) = \sum\limits_{i,j,k} {{a_{i,j,k}}} {\omega ^i}{\omega ^{'j}}{\widetilde{c}^{k}}$ with ${a_{i,j,k}} \ne 0$ be any element in $T$, and assume it commutes with $e$. Noting that $P(\omega ,{\omega ^{'}},\widetilde{c})e = eP(q\omega ,{q^{ - 1}}{\omega ^{'}}, - \widetilde{c})$, and $P(\omega ,{\omega ^{'}},\widetilde{c}) = P(q\omega ,{q^{ - 1}}{\omega ^{'}}, - \widetilde{c})$,  we get ${q^{i - j}}{( - 1)^k} = 1$. Since $q$ is not a root of unity, then $i = j$ and $k$ is even. Similarly  for $f$, we also see that $i = j$ and $k$ is even, so we get $\mathbb{C}[\omega\omega^{'} ,{\widetilde{c}}^{2}]$ is the center of $U$.
\end{proof}

To further study the structure of the algebra $U_{r, s}(osp(1,2))$, we introduce some notations.
For any $a,b \in {\mathscr F}$, denote
 $[{\omega _n},\omega _n^{'};a,b] = \frac{{{\omega _n}a - \omega _n^{'}b}}{{r - s}}$ and define two elements in ${\mathscr A}$:
\begin{eqnarray}
{Q_ - }\left( m \right) &=& \frac{{{{\left( {{r^{ - \frac{1}{2}}}{s^{\frac{1}{2}}}} \right)}^m} + {{\left( { - 1} \right)}^{m - 1}}}}{{{r^{ - \frac{1}{2}}}{s^{\frac{1}{2}}} + 1}},\\
{Q_ + }\left( m \right) &=& \frac{{{{\left( {{r^{\frac{1}{2}}}{s^{ - \frac{1}{2}}}} \right)}^m} + {{\left( { - 1} \right)}^{m - 1}}}}{{{r^{\frac{1}{2}}}{s^{ - \frac{1}{2}}} + 1}}.
\end{eqnarray}

 Also denote ${\alpha _m}\left( i \right):=\left[ {{\omega _n},\omega _n^{'};{{\left( {{r^{ - \frac{1}{2}}}{s^{\frac{1}{2}}}} \right)}^i}{Q_ - }\left( m \right),{{\left( {{r^{\frac{1}{2}}}{s^{ - \frac{1}{2}}}} \right)}^i}{Q_ + }\left( m \right)} \right]$, which is symmetric under
 $\omega\mapsto \omega', r\mapsto s$.

 \begin{lemmm} For any $p,q \in \mathbb{N}$, we have
 \begin{equation}
{e^p}{f^q} = \sum\limits_{t = 0}^{\min (p,q)} {{{( - 1)}^{(p - t)(q - t)}}} {f^{q - t}}Z(\omega ,{\omega ^{'}};p,q;t){e^{p - t}},
\end{equation}
where $Z(\omega ,{\omega ^{'}};p,q;0)=1$, and for any integer $t > 0$,
\begin{equation*}
Z(\omega ,{\omega ^{'}};p,q;t) = \sum\limits_{(*)} {\textit{sign}({i_1}} , \ldots ,{i_t}){\alpha _{\max (p,q) - t + 1}}({i_1}){\alpha _{\max (p,q) - t + 2}}({i_2})\cdots{\alpha _{\max (p,q)}}({i_t}),
\end{equation*}
where $(*)$ means summing over all $0 \le {i_1} \le  \cdots \le {i_t} \le \min (p,q) - t$, and
\begin{equation*}
\textit{sign}({i_1}, \ldots ,{i_t}) = {( - 1)^{t(\min (p,q) - t) - ({i_1} +  \ldots + {i_t})}}.
\end{equation*}
\end{lemmm}
\begin{proof}
This lemma is a special case of \cite[Lemma 4.2]{Sh} for $n=1$.
\end{proof}
\begin{theo}\label{main}
   When $r{s^{ - 1}}$ is  a root of unity, $U$ has no zero-divisors.
   \end{theo}
\begin{proof}
Using the PBW theorem, let us order the monomials with respect to their exponents lexicographically.
Then we prove the theorem by using the total order. 

Let $F$ and $G$ be any two elements in $U$, each is a sum of monomials in $f, \omega, \omega', e$. According to the $f$-degree, we can write
\begin{eqnarray*}
F &=& {F^N} + {F^{N - 1}} +  \cdots  + {F^0},
\end{eqnarray*}
where ${F^i}$ denotes the summand containing $f^i$. Explicitly, we can write
\begin{equation*}
\begin{split}
&{f^N} \Bigg{[} {{\omega ^{{s_{N{}_1}}}}{\omega ^{^{'}}}^{t_{{N_1}}^1}\left( {{a_{N,{s_{N{}_1}},t_{{N_1}}^1,u_{{N_1}}^{{1_1}}}}{e^{u_{{N_1}}^{{1_1}}}} + {a_{N,{s_{N{}_1}},t_{{N_1}}^1,u_{{N_1}}^{{1_2}}}}{e^{u_{{N_1}}^{{1_2}}}} +  \cdots  + {a_{N,{s_{N{}_1}},t_{{N_1}}^1,u_{{N_1}}^{{1_{{i_{N{}_1}}}}}}}{e^{u_{{N_1}}^{{1_{{i_{N{}_1}}}}}}}} \right)}  \\
   &\ +{\omega ^{{s_{N{}_1}}}}{\omega ^{^{'}}}^{t_{{N_1}}^2}\left( {{a_{N,{s_{N{}_1}},t_{{N_1}}^2,u_{{N_1}}^{{2_1}}}}{e^{u_{{N_1}}^{{2_1}}}} + {a_{N,{s_{N{}_1}},t_{{N_1}}^2,u_{{N_1}}^{{2_2}}}}{e^{u_{{N_1}}^{{2_2}}}} +  \cdots  + {a_{N,{s_{N{}_1}},t_{{N_1}}^2,u_{{N_1}}^{{2_{{i_{N{}_1}}}}}}}{e^{u_{{N_1}}^{{2_{{i_{N{}_1}}}}}}}} \right) +  \cdots \\
   &\ +{\omega ^{{s_{N{}_1}}}}{\omega ^{^{'}}}^{t_{{N_1}}^{{i_{N{}_1}}}}\left( {{a_{N,{s_{N{}_1}},t_{{N_1}}^{{i_{N{}_1}}},u_{{N_1}}^{{i_{N{}_{{1_1}}}}}}}{e^{u_{{N_1}}^{{i_{N{}_{{1_1}}}}}}} +   \cdots  +
   {a_{N,{s_{N{}_1}},t_{{N_1}}^{{i_{N{}_1}}},u_{{N_1}}^{{i_{N{}_{{1_{{i_{N{}_1}}}}}}}}}}{e^{u_{{N_1}}^{{i_{N{}_{{1_{{i_{N{}_1}}}}}}}}}}} \right)
\\
  &\ +{\omega ^{{s_{N{}_1}}}}{\omega ^{^{'}}}^{t_{{N_2}}^1}\left( {{a_{N,{s_{N{}_1}},t_{{N_2}}^1,u_{{N_2}}^{{1_1}}}}{e^{u_{{N_2}}^{{1_1}}}} + {a_{N,{s_{N{}_1}},t_{{N_2}}^1,u_{{N_2}}^{{1_2}}}}{e^{u_{{N_2}}^{{1_2}}}} +
    \cdots  + {a_{N,{s_{N{}_1}},t_{{N_2}}^1,u_{{N_2}}^{{1_{{i_{N{}_2}}}}}}}{e^{u_{{N_2}}^{{1_{{i_{N{}_2}}}}}}}} \right)\\
&\ +{\omega ^{{s_{N{}_1}}}}{\omega ^{^{'}}}^{t_{{N_2}}^2}\left( {{a_{N,{s_{N{}_1}},t_{{N_2}}^2,u_{{N_2}}^{{2_1}}}}{e^{u_{{N_2}}^{{2_1}}}} + {a_{N,{s_{N{}_1}},t_{{N_2}}^2,u_{{N_2}}^{{2_2}}}}{e^{u_{{N_2}}^{{2_2}}}} +  \cdots  + {a_{N,{s_{N{}_1}},t_{{N_2}}^2,u_{{N_2}}^{{2_{{i_{N{}_2}}}}}}}{e^{u_{{N_2}}^{{2_{{i_{N{}_2}}}}}}}}\right) +  \cdots  \\
&\ +{\omega ^{{s_{N{}_1}}}}{\omega ^{^{'}}}^{t_{{N_2}}^{{i_{N{}_2}}}}\left( {{a_{N,{s_{N{}_1}},t_{{N_2}}^{{i_{N{}_2}}},u_{{N_2}}^{_{{i_{N{}_{{2_1}}}}}}}}{e^{u_{{N_2}}^{_{{i_{N{}_{{2_1}}}}}}}} + 
 \cdots  + {a_{N,{s_{N{}_1}},t_{{N_2}}^{{i_{N{}_2}}},u_{{N_2}}^{{i_{N{}_{{2_{{i_{N{}_2}}}}}}}}}}{e^{u_{{N_2}}^{{i_{N{}_{{2_{{i_{N{}_2}}}}}}}}}}} \right)\\
&\  +  \cdots  + {\omega ^{{s_{N{}_1}}}}{\omega ^{^{'}}}^{t_{{N_p}}^1}\left( {{a_{N,{s_{N{}_1}},t_{{N_p}}^1,u_{{N_p}}^{_{{1_1}}}}}{e^{u_{{N_p}}^{_{{1_1}}}}} + {a_{N,{s_{N{}_1}},t_{{N_p}}^1,u_{{N_p}}^{_{{1_2}}}}}{e^{u_{{N_p}}^{{1_2}}}} +  \cdots  + {a_{N,{s_{N{}_1}},t_{{N_p}}^1,u_{{N_p}}^{_{{1_{{i_{{N_p}}}}}}}}}{e^{u_{{N_p}}^{_{{1_{{i_{{N_p}}}}}}}}}} \right)\\
&\  +  {\omega ^{{s_{N{}_1}}}}{\omega ^{^{'}}}^{t_{{N_p}}^2}\left( {{a_{N,{s_{N{}_1}},t_{{N_p}}^2,u_{{N_p}}^{_{{2_1}}}}}{e^{u_{{N_p}}^{_{{2_1}}}}} + {a_{N,{s_{N{}_1}},t_{{N_p}}^2,u_{{N_p}}^{_{{2_2}}}}}{e^{u_{{N_p}}^{{2_2}}}} +  \cdots  + {a_{N,{s_{N{}_1}},t_{{N_p}}^2,u_{{N_p}}^{_{{2_{{i_{{N_p}}}}}}}}}{e^{u_{{N_p}}^{_{{2_{{i_{{N_p}}}}}}}}}} \right)+  \cdots \\
&\ + {\omega ^{{s_{{N_1}}}}}{\omega ^{'t_{{N_p}}^{{i_{{N_p}}}}}}\left( {{a_{N,{s_{{N_1}}},t_{{N_p}}^{{i_{{N_p}}}},u_{{N_p}}^{{i_{_{{N_{{p_1}}}}}}}}}{e^{u_{{N_p}}^{{i_{_{{N_{{p_1}}}}}}}}} + 
\cdots  + {a_{N,{s_{{N_1}}},t_{{N_p}}^{{i_{{N_p}}}},u_{{N_p}}^{{i_{_{{N_{{p_{{i_{{N_p}}}}}}}}}}}}}{e^{u_{{N_p}}^{{i_{_{{N_{{p_{{i_{{N_p}}}}}}}}}}}}}} \right)\\
&\  + \textit{lower-order\ terms\ in\ } \omega\Bigg{]},
\end{split}
\end{equation*}
where the exponents $t_{{N_l}}^1 > t_{{N_l}}^2 >  \cdots  > t_{{N_l}}^{{i_{N{}_l}}}$ and $u_{{N_l}}^{_{{j_1}}} > u_{{N_l}}^{_{{j_2}}} >  \cdots  > u_{{N_l}}^{_{{j_{{i_{N{}_l}}}}}}(l = 1,2, \ldots ,p ; j = 1,2, \ldots ,{i_{N{}_l}})$. Clearly, the sum ${F^{N - 1}} +  \cdots  + {F^0}$ denotes the lower degree terms of $F$ in the lexicographic order. 

Similarly, we write the element $G$
\begin{eqnarray*}
G &=& {G^R} + {G^{R - 1}} +  \cdots  + {G^0},
\end{eqnarray*}
where ${G^R}$ denotes the summand containing $f^R$ given by
\begin{equation*}
\begin{split}
&{f^R} \Bigg{[} {{\omega ^{{s_{R{}_1}}}}{\omega ^{^{'}}}^{t_{{R_1}}^1}\left( {{b_{R,{s_{R{}_1}},t_{{R_1}}^1,u_{{R_1}}^{{1_1}}}}{e^{u_{{R_1}}^{{1_1}}}} + {b_{N,{s_{R{}_1}},t_{{R_1}}^1,u_{{R_1}}^{{1_2}}}}{e^{u_{{R_1}}^{{1_2}}}} +  \cdots  + {b_{N,{s_{R{}_1}},t_{{R_1}}^1,u_{{R_1}}^{{1_{{i_{R{}_1}}}}}}}{e^{u_{{R_1}}^{{1_{{i_{R{}_1}}}}}}}} \right)}  \\
   & +{\omega ^{{s_{R{}_1}}}}{\omega ^{^{'}}}^{t_{{R_1}}^2}\left( {{b_{R,{s_{R{}_1}},t_{{R_1}}^2,u_{{R_1}}^{{2_1}}}}{e^{u_{{R_1}}^{{2_1}}}} + {b_{R,{s_{R{}_1}},t_{{R_1}}^2,u_{{R_1}}^{{2_2}}}}{e^{u_{{R_1}}^{{2_2}}}} +  \cdots  + {b_{R,{s_{R{}_1}},t_{{N_1}}^2,u_{{R_1}}^{{2_{{i_{R{}_1}}}}}}}{e^{u_{{R_1}}^{{2_{{i_{R{}_1}}}}}}}} \right) +  \cdots  \\
   &+{\omega ^{{s_{R{}_1}}}}{\omega ^{^{'}}}^{t_{{R_1}}^{{i_{R{}_1}}}}\left( {{b_{R,{s_{R{}_1}},t_{{R_1}}^{{i_{R{}_1}}},u_{{R_1}}^{{i_{R{}_{{1_1}}}}}}}{e^{u_{{R_1}}^{{i_{R{}_{{1_1}}}}}}} + {b_{R,{s_{R{}_1}},t_{{R_1}}^{{i_{R{}_1}}},u_{{R_1}}^{{i_{R{}_{{1_2}}}}}}}{e^{u_{{R_1}}^{{i_{R{}_{{1_2}}}}}}} +  \cdots  + {b_{R,{s_{R{}_1}},t_{{R_1}}^{{i_{R{}_1}}},u_{{R_1}}^{{i_{R{}_{{1_{{i_{R{}_1}}}}}}}}}}{e^{u_{{R_1}}^{{i_{R{}_{{1_{{i_{R{}_1}}}}}}}}}}} \right)\\
   &+{\omega ^{{s_{R{}_1}}}}{\omega ^{^{'}}}^{t_{{R_2}}^1}\left( {{b_{R,{s_{R{}_1}},t_{{R_2}}^1,u_{{R_2}}^{{1_1}}}}{e^{u_{{R_2}}^{{1_1}}}} + {b_{R,{s_{R{}_1}},t_{{R_2}}^1,u_{{R_2}}^{{1_2}}}}{e^{u_{{R_2}}^{{1_2}}}} +  \cdots  + {b_{R,{s_{R{}_1}},t_{{R_2}}^1,u_{{R_2}}^{{1_{{i_{R{}_2}}}}}}}{e^{u_{{R_2}}^{{1_{{i_{R{}_2}}}}}}}} \right)\\
 &+{\omega ^{{s_{R{}_1}}}}{\omega ^{^{'}}}^{t_{{R_2}}^2}\left( {{b_{R,{s_{R{}_1}},t_{{R_2}}^2,u_{{R_2}}^{{2_1}}}}{e^{u_{{R_2}}^{{2_1}}}} + {b_{R,{s_{R{}_1}},t_{{R_2}}^2,u_{{R_2}}^{{2_2}}}}{e^{u_{{R_2}}^{{2_2}}}} +  \cdots  + {b_{R,{s_{R{}_1}},t_{{R_2}}^2,u_{{R_2}}^{{2_{{i_{R{}_2}}}}}}}{e^{u_{{R_2}}^{{2_{{i_{R{}_2}}}}}}}}\right) +  \cdots  \\
&+{\omega ^{{s_{R{}_1}}}}{\omega ^{^{'}}}^{t_{{R_2}}^{{i_{R{}_2}}}}\left( {{b_{R,{s_{R{}_1}},t_{{R_2}}^{{i_{R{}_2}}},u_{{R_2}}^{_{{i_{R{}_{{2_1}}}}}}}}{e^{u_{{R_2}}^{_{{i_{R{}_{{2_1}}}}}}}} + {b_{R,{s_{R{}_1}},t_{{R_2}}^{{i_{R{}_2}}},u_{{R_2}}^{{i_{R{}_{{2_2}}}}}}}{e^{u_{{R_2}}^{{i_{R{}_{{2_2}}}}}}} +  \cdots  + {b_{R,{s_{R{}_1}},t_{{R_2}}^{{i_{R{}_2}}},u_{{R_2}}^{{i_{R{}_{{2_{{i_{R{}_2}}}}}}}}}}{e^{u_{{R_2}}^{{i_{R{}_{{2_{{i_{R{}_2}}}}}}}}}}} \right)\\
& +  \cdots  + {\omega ^{{s_{R{}_1}}}}{\omega ^{^{'}}}^{t_{{R_q}}^1}\left( {{a_{R,{s_{R{}_1}},t_{{R_q}}^1,u_{{R_q}}^{_{{1_1}}}}}{e^{u_{{R_q}}^{_{{1_1}}}}} + {b_{N,{s_{R{}_1}},t_{{R_q}}^1,u_{{R_q}}^{_{{1_2}}}}}{e^{u_{{R_q}}^{{1_2}}}} +  \cdots  + {b_{R,{s_{R{}_1}},t_{{R_q}}^1,u_{{R_q}}^{_{{1_{{i_{{R_q}}}}}}}}}{e^{u_{{R_q}}^{_{{1_{{i_{{R_q}}}}}}}}}} \right)\\
& +  {\omega ^{{s_{R{}_1}}}}{\omega ^{^{'}}}^{t_{{R_q}}^2}\left( {{b_{R,{s_{R{}_1}},t_{{R_q}}^2,u_{{R_q}}^{_{{2_1}}}}}{e^{u_{{R_q}}^{_{{2_1}}}}} + {b_{R,{s_{R{}_1}},t_{{R_q}}^2,u_{{R_q}}^{_{{2_2}}}}}{e^{u_{{R_q}}^{{2_2}}}} +  \cdots  + {b_{R,{s_{R{}_1}},t_{{R_q}}^2,u_{{R_q}}^{_{{2_{{i_{{R_q}}}}}}}}}{e^{u_{{R_q}}^{_{{2_{{i_{{R_q}}}}}}}}}} \right)+  \cdots
\end{split}
\end{equation*}
\begin{equation*}
\begin{split}
& +{\omega ^{{s_{R{}_1}}}}{\omega ^{^{'}}}^{t_{{R_q}}^{{i_{R{}_q}}}}\left( {{b_{R,{s_{R{}_1}},t_{{R_q}}^{{i_{R{}_q}}},u_{{R_q}}^{{i_{R{}_{{q_1}}}}}}}{e^{u_{{R_q}}^{{i_{R{}_{{q_1}}}}}}} + {b_{R,{s_{R{}_1}},t_{{R_q}}^{{i_{R{}_q}}},u_{{R_q}}^{{i_{R{}_{{q_2}}}}}}}{e^{u_{{R_q}}^{{i_{R{}_{{q_2}}}}}}} +  \cdots  + {b_{R,{s_{R{}_1}},t_{{R_q}}^{{i_{R{}_q}}},u_{{R_1}}^{{i_{R{}_{{q_{{i_{R{}_q}}}}}}}}}}{e^{u_{{R_1}}^{{i_{R{}_{{q_{{i_{R{}_q}}}}}}}}}}} \right)\\
&+\textit{lower-order\ terms\ with\ respect\ to\ } \omega \Bigg{]}.
\end{split}
\end{equation*}
Here the exponents $t_{{R_l}}^1 > t_{{R_l}}^2 >  \cdots  > t_{{R_l}}^{{i_{R{}_l}}}$ and $u_{{R_l}}^{_{{j_1}}} > u_{{R_l}}^{_{{j_2}}} >  \cdots  > u_{{R_l}}^{_{{j_{{i_{R{}_l}}}}}}(l = 1,2, \ldots ,q ; j = 1,2, \ldots ,{i_{R{}_l}})$, and the lower ordered terms are
${G^{R - 1}} +  \cdots  + {G^0}$ for $G$.

From Lemma 3.1, it follows that
 \begin{equation*}
 \begin{split}
 &({f^N}{\omega ^{{s_{N{}_1}}}}{\omega ^{^{'}}}^{t_{{N_1}}^1}{e^{u_{{N_1}}^{{1_1}}}})({f^R}{\omega ^{{s_{R{}_1}}}}{\omega ^{^{'}}}^{t_{{R_1}}^1}{e^{u_{{R_1}}^{{1_1}}}})\\
 &={f^N}{\omega ^{{s_{N{}_1}}}}{\omega ^{^{'}}}^{t_{{N_1}}^1}\left( {\sum\limits_{t = 0}^{\min (u_{{N_1}}^{{1_1}},R)} {{{( - 1)}^{(u_{{N_1}}^{{1_1}} - t)(R - t)}}} {f^{R - t}}Z(\omega ,{\omega ^{'}};u_{{N_1}}^{{1_1}},R;t){e^{u_{{N_1}}^{{1_1}} - t}}} \right){\omega ^{{s_{R{}_1}}}}{\omega ^{^{'}}}^{t_{{R_1}}^1}{e^{u_{{R_1}}^{{1_1}}}}  \\
   &= {( - 1)^{u_{{N_1}}^{{1_1}}R}}{f^N}{\omega ^{{s_{N{}_1}}}}{\omega ^{^{'}}}^{t_{{N_1}}^1}{f^R}Z(\omega ,{\omega ^{'}};u_{{N_1}}^{{1_1}},R;0){e^{u_{{N_1}}^{{1_1}}}}{\omega ^{{s_{R{}_1}}}}{\omega ^{^{'}}}^{t_{{R_1}}^1}{e^{u_{{R_1}}^{{1_1}}}}
   + \textit{lower-order\ terms\ on } f
\\
   &={( - 1)^{u_{{N_1}}^{{1_1}}R}}{q^{R(t_{{N_1}}^1 - {s_{N{}_1}}) + u_{{N_1}}^{{1_1}}(t_{{R_1}}^1 - {s_{R{}_1}})}}{f^{N + R}}{\omega ^{{s_{N{}_1}} + {s_{R{}_1}}}}{\omega ^{^{'}}}^{t_{{N_1}}^{{1_1}} + t_{{R_1}}^1}{e^{u_{{N_1}}^{{1_1}} + u_{{R_1}}^{{1_1}}}} + \textit{lower-order\ terms on\ } f.
\end{split}
\end{equation*}
Then we have
 \begin{equation*}
 \begin{split}
 FG &= {F^N}{G^R} + \mathrm{lower\, ordered \, terms}\\
&={f^N}{\omega ^{{s_{N{}_1}}}}{\omega ^{^{'}}}^{t_{{N_1}}^1}{a_{N,{s_{N{}_1}},t_{{N_1}}^1,u_{{N_1}}^{{1_1}}}}
{e^{u_{{N_1}}^{{1_1}}}}{f^R}{\omega ^{{s_{R{}_1}}}}{\omega ^{^{'}}}^{t_{{R_1}}^1}{b_{R,{s_{R{}_1}},t_{{R_1}}^1,u_{{R_1}}^{{1_1}}}}
{e^{u_{{R_1}}^{{1_1}}}}\\
&\quad+ \mathrm{lower\, ordered \, terms} \\
&={( - 1)^{u_{{N_1}}^{{1_1}}R}}{q^{R(t_{{N_1}}^1 - {s_{N{}_1}}) + u_{{N_1}}^{{1_1}}(t_{{R_1}}^1 - {s_{R{}_1}})}}{a_{N,{s_{N{}_1}},t_{{N_1}}^1,u_{{N_1}}^{{1_1}}}}
{b_{R,{s_{R{}_1}},t_{{R_1}}^1,u_{{R_1}}^{{1_1}}}}\times\\
&\quad\times{f^{N + R}}{\omega ^{{s_{N{}_1}} + {s_{R{}_1}}}}{\omega ^{^{'}}}^{t_{{N_1}}^1 + t_{{R_1}}^1}{e^{u_{{N_1}}^{{1_1}} + u_{{R_1}}^{{1_1}}}}\\
&\quad+ \mathrm{lower\, ordered \, terms}.
\end{split}
\end{equation*}
Note that the leading coefficient
$${q^{R(t_{{N_1}}^1 - {s_{N{}_1}}) + u_{{N_1}}^{{1_1}}(t_{{R_1}}^1 - {s_{R{}_1}})}}{a_{N,{s_{N{}_1}},t_{{N_1}}^1,u_{{N_1}}^{{1_1}}}}{b_{R,{s_{R{}_1}},t_{{R_1}}^1,u_{{R_1}}^{{1_1}}}}{f^{N + R}}{\omega ^{{s_{N{}_1}} + {s_{R{}_1}}}}{\omega ^{^{'}}}^{t_{{N_1}}^1 + t_{{R_1}}^1}{e^{u_{{N_1}}^{{1_1}} + u_{{R_1}}^{{1_1}}}}$$
is nonzero, the leading summand $(FG)^{N+R}$ of $FG$ is nonzero, so $FG\neq 0$.
\end{proof}
For the rest of this section, suppose $\l$ is an integer larger than $2$ and $q$ is a primitive ${\l ^{th}}$-root of unity. The integer $L$ is the smallest even multiple of $\l$ (that is, $L=\l$ if $\l$ is even and $L=2l$ if $\l$ is odd).
The integer ${\l'}$ is ${L}/2$ if $\l$ is twice an odd integer and $L$ otherwise. Then ${q^{'}} =  - q$ is a primitive ${\l ^{th}}$ root of unity, and ${q^L} = 1$. Evaluation of $(2.5)$ for $m=L$ gives ${f^L}e - e{f^L} = 0$, so $e$ commutes with ${f^L}$. Applying $(2.6)$, we can get $f$ commutes with ${e^L}$.
Both $E \equiv {e^L}$ and $F \equiv {f^L}$ also commute with $\omega ,{\omega ^{'}}$, consequently, $E$ and $F$ are central elements. Moreover, $\omega {\omega ^{'}},{\widetilde{c}}^{2},{\omega ^{ \pm l}},{\omega ^{' \pm l}}$ are also central elements. The following result determines the center of $U$ when $r{s^{ - 1}}$  is a root of unity.
\begin{theo}\label{main}
   When $r{s^{ - 1}}$ is  a root of unity, the center of $U$  contains
   \begin{equation*}
   \mathbb{C}[\omega {\omega ^{'}},{\widetilde{c}}^{2},{\omega ^{ \pm 1}},{\omega ^{{'} \pm 1}},F] + \mathbb{C}[\omega {\omega ^{'}},{\widetilde{c}}^{2},{\omega ^{ \pm 1}},{\omega ^{{'} \pm 1}},E],
   \end{equation*}

   $(1)$ If $l$ is odd, this is the center of the algebra. 

   $(2)$ If $l$ is even, the center is the free module on  $$\mathbb{C}[\omega {\omega ^{'}},{\widetilde{c}}^{2},{\omega ^{ \pm 1}},{\omega ^{{'} \pm 1}},F] + \mathbb{C}[\omega {\omega ^{'}},{\widetilde{c}}^{2},{\omega ^{ \pm 1}},{\omega ^{{'} \pm 1}},E]$$
   with basis $1$, ${\omega ^{{\textstyle{1 \over 2}}}}\widetilde{c}$, ${\omega ^{'{\textstyle{1 \over 2}}}}\widetilde{c}$.
   \end{theo}
\begin{proof}
The commutative subalgebras $T[E]$ and $T[F]$ are free polynomial algebras. We can use $F$ and $E$ to refine the previously obtained decomposition of $U$ as an $T$-module: $U = {U_ + } + {U_ - }$ where ${U_ + }$ is the free $T[F]$-module with basis $1,f, \cdots ,{f^{L - 1}}$ and ${U_ - }$ the free $T[E]$-module with basis $1,e, \cdots ,{e^{L - 1}}$ (the sum ${U_ + } + {U_ - }$ is not direct, because ${U_ + } \cap {U_ - } = T$).
The monomial $e^{m}$ commutes with $\omega$ and ${\omega ^{'}}$ only if $m$ is a multiple of $l$ and with $\widetilde{c}$ only if $m$ is even. The same is true for powers of $f$.
Hence the commutant of $\omega$, ${\omega ^{'}}$ and $\widetilde{c}$ in $U$ is $T[E] + T[F]$.

Now we use the same argument of the previous section. Any element of $T[E] + T[F]$ can be expanded in powers of $\omega$, ${\omega ^{'}}$ and $\widetilde{c}$. In the expanded formula of a central element, a monomial ${\omega ^i}{\omega ^{'j}}{\widetilde{c}}^{k}$ can appear with non-zero coefficient only if ${q^{i-j}}{( - 1)^k} = 1$. This happens for instance if $i-j$ is a multiple of $l$ and $k$ is even. If $l$ is odd, this is the only possibility, i.e., $(1)$. If $l$ is even, in addition to the set of solution that $i-j$ is a multiple of $l$ and $k$ is even, there is another set of solution, that is, $i-j$ is odd multiple of $l/2$, and $k$ is  odd. In this circumstance, the center is the free module on $\mathbb{C}[\omega {\omega ^{'}},{\widetilde{c}}^{2},{\omega ^{ \pm 1}},{\omega ^{{'} \pm 1}},F] + \mathbb{C}[\omega {\omega ^{'}},{\widetilde{c}}^{2},{\omega ^{ \pm 1}},{\omega ^{{'} \pm 1}},E]$ with basis $1$, ${\omega ^{{\textstyle{1 \over 2}}}}\widetilde{c}$, ${\omega ^{'{\textstyle{1 \over 2}}}}\widetilde{c}$.
\end{proof}
This theorem gives an explicit description of the center, but there are some drawbacks associated with the multiplication structure. For instance, though $E$ and $F$ are described abstractly as above, we still need an exact expression for $FE$. This is considered in Sect. 5.

\section{\textbf{Dickson polynomials}}

The Dickson polynomial was originally introduced over finite fields, and was known to be a deformation of the Chebychev polynomial (cf. {\cite{LMT}}).

\begin{defini}
The Dickson polynomial ${D_n}(u,a)$ of the first kind of degree $n$ in the indeterminate $u$ with parameter $a \in \mathbb{C}$ is given by
\begin{equation*}
{D_n}(u,a) = \sum\limits_{i = 0}^{\left[ {{\textstyle{m \over 2}}} \right]} {\frac{m}{{m - i}}}\dbinom{m-i}{i} {( - a)^i}{u^{m - 2i}},
\end{equation*}
where ${\left[ {{\textstyle{m \over 2}}} \right]}$ denotes the largest integer $ \le {\textstyle{m \over 2}}$. For $n=0$, set ${D_0}(u,a) = 2$, and for $n=1$, ${D_1}(u,a) = u$.
\end{defini}
We remark that ${D_m}(u,1) = 2{P_m}\left( {\frac{u}{2}} \right)$ with $u = 2\cos \theta$ is the Chebychev polynomial.
\begin{lem} \cite{Ca} (i) The Dickson polynomial is determined by the recursive relation
 \begin{equation*}
{D_{m + 2}}(u,a) = u{D_{m + 1}}(u,a) - a{D_m}(u,a),  \qquad n\geq0
\end{equation*}
with the initial condition
\begin{equation*}
{D_0}(u,a) = 2, \ {D_1}(u,a) = u.
\end{equation*}
(ii) ${D_{m}}(u,a)$ has the following generating function
\begin{equation*}
\sum\limits_{m \ge 0} {{D_m}(u,a)} {z^m} = \frac{{2 - uz}}{{1 - uz + a{z^2}}}.
\end{equation*}
\end{lem}

 If $u$ is an indeterminate and $m$ a positive integer, we claim that ${u^m} + {\left( { - \frac{1}{u}} \right)^m}$ is a monic polynomial of order $m$ in
 the variable $u - \frac{1}{u} = {b^{ - 1}}v\widetilde{c}$. In fact, ${u^m} + {\left( { - \frac{1}{u}} \right)^m}$ is the Dickson polynomial ${D_m}({b^{ - 1}}v\widetilde{c}, - 1)$. This is seen by considering the generating function:
$$\sum\limits_{m \ge 0} {{z^m}\left( {{u^m} + {{\left( { - \frac{1}{u}} \right)}^m}} \right)}  = \frac{1}{{1 - uz}} + \frac{1}{{1 + \frac{1}{u}z}} = \frac{{2 - {b^{ - 1}}v\widetilde{c}z}}{{1 - {b^{ - 1}}v\widetilde{c}z - {z^2}}} = \sum\limits_{m \ge 0} {{z^m}} {D_m}({b^{ - 1}}v\widetilde{c}, - 1).$$
Thus, ${D_m}({b^{ - 1}}v\widetilde{c}, - 1) = {u^m} + {\left( { - \frac{1}{u}} \right)^m}$.

We can check that ${D_2}({b^{ - 1}}v\widetilde{c}, - 1) = {b^{ - 2}}{v^2}{\widetilde{c}}^{2} + 2 = C$, and ${D_{2m}}({b^{ - 1}}v\widetilde{c}, - 1)$ is a polynomial in $C$. In fact,
\begin{equation*}
\begin{split}
 \sum\limits_{m \ge 0} {{z^m}} {D_{2m}}({b^{ - 1}}v\widetilde{c}, - 1)
 &= \sum\limits_{m \ge 0} {{z^m}\left( {{u^{2m}} + {{\left( { - \frac{1}{u}} \right)}^{2m}}} \right)}  \\
 &= {\sum\limits_{m \ge 0} {(z{u^2})} ^m} + {\sum\limits_{m \ge 0} {\left( {z\frac{1}{{{u^2}}}} \right)} ^m}
\\
    &= \frac{1}{{1 - {u^2}z}} + \frac{1}{{1 - \frac{1}{{{u^2}}}z}} = \frac{{2 - \left( {{u^2} + \frac{1}{{{u^2}}}} \right)z}}{{1 - \left( {{u^2} + \frac{1}{{{u^2}}}} \right)z + {z^2}}}\\
&=\frac{{2 - \left( {{b^{ - 2}}{v^2}{\widetilde{c}}^{2} + 2} \right)z}}{{1 - \left( {{b^{ - 2}}{v^2}{\widetilde{c}}^{2} + 2} \right)z + {z^2}}} \\
&= \frac{{2 - Cz}}{{1 - Cz + {z^2}}} = \sum\limits_{m \ge 0} {{z^m}{T_m}} (C,1).
 \end{split}
 \end{equation*}

 In particular, ${T_m}(C,1) = {u^{2m}} + {\left( { - \frac{1}{u}} \right)^{2m}}.$

Comparing the generation functions, we obtain that
$${D_m}(i{b^{ - 1}}v\widetilde{c}, - 1) = i{T_m}({b^{ - 1}}v\widetilde{c},1),\quad {T_m}(i{b^{ - 1}}v\widetilde{c},1) = i{D_m}({b^{ - 1}}v\widetilde{c}, - 1).$$
Actually,
\begin{equation*}
\begin{split}
\sum\limits_{m \ge 0} {{z^m}} {D_{m}}(i{b^{ - 1}}v\widetilde{c}, - 1) &= \frac{{2 - i{b^{ - 1}}v\widetilde{c}z}}{{1 - i{b^{ - 1}}v\widetilde{c}z - {z^2}}}. \\
\sum\limits_{m \ge 0} {{z^m}{i^m}} {T_m}\left( {{b^{ - 1}}v\widetilde{c},1} \right)& = {\sum\limits_{m \ge 0} {(iz)} ^m}{T_m}\left( {{b^{ - 1}}v\widetilde{c},1} \right)\\
& = \frac{{2 - i{b^{ - 1}}v\widetilde{c}z}}{{1 - i{b^{ - 1}}v\widetilde{c}z + {{(iz)}^2}}}= \frac{{2 - i{b^{ - 1}}v\widetilde{c}z}}{{1 - i{b^{ - 1}}v\widetilde{c}z - {z^2}}}.
\end{split}
\end{equation*}
Thus, ${D_m}(i{b^{ - 1}}v\widetilde{c}, - 1) = i{T_m}({b^{ - 1}}v\widetilde{c},1)$. The other equation can be verified in a similar manner.

We remark that $\frac{{{D_{2m + 1}}({b^{ - 1}}v\widetilde{c}, - 1)}}{{{b^{ - 1}}v\widetilde{c}}}$ is a polynomial in $C$. In fact,
\begin{equation*}
\begin{split}
\sum\limits_{m \ge 0} {{z^m}\frac{{{D_{2m + 1}}({b^{ - 1}}v\widetilde{c}, - 1)}}{{{b^{ - 1}}v\widetilde{c}}}}  &= \sum\limits_{m \ge 0} {{z^m}} \left( {{u^{2m}} + {{\left( { - {u^{ - 1}}} \right)}^{2m}}} \right) = \frac{{(u - {u^{ - 1}}) + (u - {u^{ - 1}})z}}{{(u - {u^{ - 1}})\left( {1 - ({u^2} + {u^{ - 2}})z + {z^2}} \right)}}
\\
 & = \frac{{1 + z}}{{1 - Cz + {z^2}}} = \sum\limits_{m \ge 0} {{z^m}} {R_m}(C,1).
 \end{split} \end{equation*}

\section{\textbf{Scasimir, generators and relations }}
In Proposition $2.2$, when $m = {l'}$, we have
\begin{equation}\label{e:prodrel}
 \mathop \Pi \limits_{n = 0}^{{l'} - 1} \left( {\widetilde{c} - {q^{'n}}{r^{{\textstyle{1 \over 2}}}}\omega  + {q^{' - n}}{s^{{\textstyle{1 \over 2}}}}{\omega ^{{'}}}} \right) = \varepsilon ({l'})( - \eta ){}^{{l'}}{f^{{l'}}}{e^{{l'}}}.
 \end{equation}
We know ${s^{ - {\textstyle{1 \over 4}}}}{\omega ^{'}}^{ - {\textstyle{1 \over 2}}} = v$, ${r^{-{\textstyle{1 \over 4}}}}{\omega ^{-{\textstyle{1 \over 2}}}}{s^{ - {\textstyle{1 \over 4}}}}{\omega ^{'}}^{ - {\textstyle{1 \over 2}}}\widetilde{c} = {b^{ - 1}}v\widetilde{c}=u - {u^{ - 1}}$. Then
\begin{equation*}
 \begin{split}
 \mathop \Pi \limits_{n = 0}^{{l'} - 1} &\left( {\widetilde{c} - {q^{'n}}{r^{{\textstyle{1 \over 2}}}}\omega  + {q^{' - n}}{s^{{\textstyle{1 \over 2}}}}{\omega ^{'}}} \right)\\
&=\mathop \Pi \limits_{n = 0}^{{l'} - 1} \left[ {\widetilde{c} - {r^{{\textstyle{1 \over 4}}}}{s^{{\textstyle{1 \over 4}}}}{\omega ^{{\textstyle{1 \over 2}}}}{\omega ^{'}}^{{\textstyle{1 \over 2}}}\left( {{q^{'n}}{r^{{\textstyle{1 \over 4}}}}{s^{ - {\textstyle{1 \over 4}}}}{\omega ^{{\textstyle{1 \over 2}}}}{\omega ^{'}}^{ - {\textstyle{1 \over 2}}} - {q^{' - n}}{r^{ - {\textstyle{1 \over 4}}}}{s^{{\textstyle{1 \over 4}}}}{\omega ^{ - {\textstyle{1 \over 2}}}}{\omega ^{'}}^{{\textstyle{1 \over 2}}}} \right)} \right]  \\
   &= \mathop \Pi \limits_{n = 0}^{{l'} - 1} {r^{{\textstyle{1 \over 4}}}}{s^{{\textstyle{1 \over 4}}}}{\omega ^{{\textstyle{1 \over 2}}}}{\omega ^{'}}^{{\textstyle{1 \over 2}}}\mathop \Pi \limits_{n = 0}^{{l'} - 1} \left( {{r^{ - {\textstyle{1 \over 4}}}}{\omega ^{ - {\textstyle{1 \over 2}}}}{s^{ - {\textstyle{1 \over 4}}}}{\omega ^{'}}^{ - {\textstyle{1 \over 2}}}\widetilde{c} - {q^{'n}}{r^{{\textstyle{1 \over 4}}}}{\omega ^{{\textstyle{1 \over 2}}}}{s^{ - {\textstyle{1 \over 4}}}}{\omega ^{'}}^{ - {\textstyle{1 \over 2}}} + {q^{' - n}}{r^{ - {\textstyle{1 \over 4}}}}{\omega ^{ - {\textstyle{1 \over 2}}}}{s^{{\textstyle{1 \over 4}}}}{\omega ^{'}}^{{\textstyle{1 \over 2}}}} \right)\\
   &= \mathop \Pi \limits_{n = 0}^{{l'} - 1} b{v^{ - 1}}\mathop \Pi \limits_{n = 0}^{{l'} - 1} \left( {u - {u^{ - 1}} - b{q^{'n}}v + {b^{ - 1}}{q^{' - n}}{v^{ - 1}}} \right)\\
   &= {b^{l'}}{v^{ - l'}}\mathop \Pi \limits_{n = 0}^{{l'} - 1} {u^{ - 1}}(u - b{q^{'n}}v)(u + {b^{ - 1}}{q^{' - n}}{v^{ - 1}})\\
   &= {b^{l'}}{v^{ - l'}}{u^{ - l'}}\mathop \Pi \limits_{n = 0}^{{l'} - 1} (u - b{q^{'n}}v)(u + {b^{ - 1}}{q^{' - n}}{v^{ - 1}})\\
   &= {b^{l'}}{v^{ - l'}}{u^{ - l'}}\left( {{u^{l'}} - {{(bv)}^{l'}}} \right)\left( {{u^{l'}} - {{( - {b^{ - 1}}{v^{ - 1}})}^{l'}}} \right)\\
   &= {b^{l'}}{v^{ - l'}}\left( {{u^{l'}} + {{( - {u^{ - 1}})}^{l'}} - {{(bv)}^{l'}} - {{( - {b^{ - 1}}{v^{ - 1}})}^{l'}}} \right).
\end{split}
\end{equation*}

On the right hand side, we use the Dickson polynomial from Sect. 4 to rewrite $(5.1)$ as follows.
\begin{equation}
\begin{split}
 {D_{{l^{'}}}}({b^{ - 1}}v\widetilde{c}, - 1) &= {D_{{l^{'}}}}\left( {{{(rs)}^{ - {\textstyle{1 \over 4}}}}{{(\omega \omega ')}^{ - {\textstyle{1 \over 2}}}}\widetilde{c}, - 1} \right) \\
 \nonumber \quad \quad \quad \quad &={q^{{\textstyle{{{l'}} \over 2}}}}{(\omega {\omega ^{' - 1}})^{{\textstyle{{l'} \over 2}}}} + {( - 1)^{l'}}{q^{ - {\textstyle{{{l'}} \over 2}}}}{(\omega {\omega ^{' - 1}})^{ - {\textstyle{{l'} \over 2}}}} + \varepsilon ({l'}){s^{ - {\textstyle{{l'} \over 2}}}}{q^{ - {\textstyle{{l'} \over 2}}}}{( - \eta )^{l'}}{(\omega {\omega ^{'}})^{ - {\textstyle{{l'} \over 2}}}}{f^{l'}}{e^{l'}}.
\end{split}
\end{equation}
This provides a relation between $\widetilde{c}$ and the ${l^{'th}}/2$ and ${l^{'th}}$ powers of the other generators.

Based on these, we can derive function relations for the center. We divide it into cases.
\begin{theo}\label{mainly}
$(1)$ If $l$ is  not twice of an odd integer,  $l'=L$ is  even, and \eqref{e:prodrel} is a relation in the center:
\begin{equation}\label{e:rel}
   {( - 1)^{{\textstyle{L \over 2}}}}{T_{{\textstyle{L \over 2}}}}(C,1) =  - {(\omega {\omega ^{' - 1}})^{{\textstyle{L \over 2}}}} - {(\omega {\omega ^{' - 1}})^{ - {\textstyle{L \over 2}}}} - {( - 1)^{{\textstyle{L \over 2}}}}{s^{ - {\textstyle{L \over 2}}}}{\eta ^L}{(\omega {\omega ^{' - 1}})^{ - {\textstyle{L \over 2}}}}{f^L}{e^L}.
    \end{equation}
 $(2)$ If $l$ is twice of an odd integer, then $l=L=2l'$, and $l'$ is odd. Then \eqref{e:prodrel} turns into:
\begin{eqnarray}
     {(rs)^{ - {\textstyle{1 \over 4}}}}{(\omega {\omega ^{'}})^{ - {\textstyle{1 \over 2}}}}\widetilde{c}{R_{{\textstyle{{L - 2} \over 4}}}}(C,1) &=& {q^{{\textstyle{L \over 4}}}}{(\omega {\omega ^{' - 1}})^{{\textstyle{L \over 4}}}} - {q^{ - {\textstyle{L \over 4}}}}{(\omega {\omega ^{' - 1}})^{ - {\textstyle{L \over 4}}}} \\
     \nonumber &&+\; {( - 1)^{{\textstyle{{L + 2} \over 4}}}}{q^{ - {\textstyle{L \over 4}}}}{s^{ - {\textstyle{L \over 4}}}}{\eta ^{{\textstyle{L \over 2}}}}{(\omega {\omega ^{'}})^{ - {\textstyle{L \over 4}}}}{f^{{\textstyle{L \over 2}}}}{e^{{\textstyle{L \over 2}}}}.
    \end{eqnarray}

   \end{theo}
 It is easy to verify each term remains in the scenter. Consequently, $(5.3)$  is a relation in the scenter. Multiplying by ${\omega ^{{\textstyle{L \over 2}}}}$ and ${\omega ^{' - {\textstyle{L \over 2}}}}$  on both sides respectively, we can obtain two relations $(5.4)$ \& $(5.5)$ below in the center. Another relation in the center $(5.6)$ originates from an expression of ${T_{{\textstyle{L \over 2}}}}(C,1) = {D^2}_{{\textstyle{L \over 2}}}({b^{ - 1}}v\widetilde{c}, - 1)+ 2$ in connection with central elements. Combining $(5.4)$ and $(5.5)$ with $(5.6)$ yields a formula for ${f^L}{e^L}$ in the center $(5.7)$.
\begin{theo}\label{mainly}
 If $l$ is twice an odd integer,  the following relations in the center hold:
\begin{equation}
\begin{split}
 {(rs)^{ - {\textstyle{1 \over 4}}}}{\omega ^{{\textstyle{{L - 1} \over 2}}}}{\omega ^{' - {\textstyle{1 \over 2}}}}\widetilde{c}{R_{{\textstyle{{L - 2} \over 4}}}}(C,1) &= {q^{{\textstyle{L \over 4}}}}{\omega ^{{\textstyle{{3L} \over 4}}}}{\omega ^{'-{\textstyle{L \over 4}}}} - {q^{ - {\textstyle{L \over 4}}}}{\omega ^{{\textstyle{L \over 4}}}}{\omega ^{'{\textstyle{L \over 4}}}} \\
& \quad+\;{( - 1)^{{\textstyle{{L + 2} \over 4}}}}{q^{ - {\textstyle{L \over 4}}}}{s^{ - {\textstyle{L \over 4}}}}{\eta ^{{\textstyle{L \over 2}}}}{f^{{\textstyle{L \over 2}}}}{e^{{\textstyle{L \over 2}}}}{\omega ^{{\textstyle{L \over 4}}}}{\omega ^{' - {\textstyle{L \over 4}}}}.
\end{split}
\end{equation}
\begin{equation}
\begin{split}
 {(rs)^{ - {\textstyle{1 \over 4}}}}{\omega ^{ - {\textstyle{1 \over 2}}}}{\omega ^{' - {\textstyle{{L + 1} \over 2}}}}\widetilde{c}{R_{{\textstyle{L-2 \over 2}}}}(C,1) &=  {q^{{\textstyle{L \over 4}}}}{\omega ^{{\textstyle{L \over 4}}}}{\omega ^{' - {\textstyle{{3L} \over 4}}}} - {q^{ - {\textstyle{L \over 4}}}}{\omega ^{ - {\textstyle{L \over 4}}}}{\omega ^{' - {\textstyle{L \over 4}}}}\\
  &\quad+{( - 1)^{{\textstyle{{L + 2} \over 4}}}}{q^{ - {\textstyle{L \over 4}}}}{s^{ - {\textstyle{L \over 4}}}}{\eta ^{{\textstyle{L \over 2}}}}{f^{{\textstyle{L \over 2}}}}{e^{{\textstyle{L \over 2}}}}{\omega ^{ - {\textstyle{L \over 4}}}}{\omega ^{' - {\textstyle{{3L} \over 4}}}}.
  \end{split}
\end{equation}
\begin{equation}
\begin{split}
 {T_{{\textstyle{L \over 2}}}}(C,1) &=  - {(\omega {\omega ^{' - 1}})^{{\textstyle{L \over 2}}}} - {(\omega {\omega ^{' - 1}})^{ - {\textstyle{L \over 2}}}} + {s^{ - {\textstyle{L \over 2}}}}{\eta ^L}{f^L}{e^L}{(\omega {\omega ^{'}})^{ - {\textstyle{L \over 2}}}} \\
& \quad+ 2{( - 1)^{{\textstyle{{L + 2} \over 4}}}}{s^{ - {\textstyle{L \over 4}}}}{\eta ^{{\textstyle{L \over 2}}}}{f^{{\textstyle{L \over 2}}}}{e^{{\textstyle{L \over 2}}}}({\omega ^{ - {\textstyle{L \over 2}}}} + {\omega ^{' - {\textstyle{L \over 2}}}}).
\end{split}
\end{equation}
\begin{equation}
\begin{split}
  {T_{{\textstyle{L \over 2}}}}(C,1) &= {(\omega {\omega ^{' - 1}})^{{\textstyle{L \over 2}}}} + {(\omega {\omega ^{' - 1}})^{ - {\textstyle{L \over 2}}}} + {s^{ - {\textstyle{L \over 2}}}}{\eta ^L}{f^L}{e^L}{(\omega {\omega ^{'}})^{ - {\textstyle{L \over 2}}}}\\
&\quad+2{(rs)^{ - {\textstyle{1 \over 4}}}}{q^{{\textstyle{L \over 4}}}}\left( {{{(\omega {\omega ^{' - 1}})}^{{\textstyle{L \over 4}}}} + {{(\omega {\omega ^{' - 1}})}^{ - {\textstyle{L \over 4}}}}} \right){(\omega {\omega ^{'}})^{ - {\textstyle{1 \over 2}}}}\widetilde{c}{R_{{\textstyle{{L - 2} \over 4}}}}(C,1)+4.
\end{split}
\end{equation}
   \end{theo}

\bigskip

\bigskip
\centerline{\bf Acknowledgments}
\medskip
The work is supported in part by the National Natural Science Foundation of China (Grant Nos. 12171155, 12071094, 12171303)
and the Simons Foundation Grant No. 523868. The second author is also supported in part by the Science and Technology Commission of Shanghai Municipality (Grant No. 18dz2271000).

\bigskip
\bibliographystyle{amsalpha}

\begin{thebibliography}{ABC9}
\medskip



\bibitem[AAB]{AAB}
B. ~Abdesselam, D. ~Arnaudon, M. ~Bauer, {Centre and representations of ${U_q}\left( {sl(2|1)} \right)$ at roots of unity}, \textit{J. Phys. A, Gen. Math.} \textbf{30} (1997), 867--880.
\bibitem[AY]{AY}
C. ~Ai,  S. ~Yang, {Two-parameter quantum superalgebras and PBW theorem}, \textit{Algebra Colloq.} \textbf{23} (2016), 303--324.
\bibitem[AB1]{AB1}
D. ~Arnaudon, M. ~Bauer, {Polynomial relations in the centre of ${U_q}\left( {sl(N)} \right)$}, \textit{ Lett. Math. Phys.}  \textbf{30} (1994), 251--257.
\bibitem[AB2]{AB2}
 D. ~Arnaudon, M. ~Bauer, {Scasimir operator, scentre and representations of ${U_q}\left( {osp(1,2)} \right)$}, \textit{Lett. Math. Phys.}  \textbf{40} (1997),  307--320.
 \bibitem[ABF]{ABF} D. ~Arnaudon, M. ~Bauer, L. ~Frappat, {On Casimir's ghost}, \textit{Comm. Math. Phys.} \textbf{187} (1997), 429--439.
 \bibitem[ACF]{ACF} D. ~Arnaudon, C. ~Chryssomalakos, L. ~Frappat, {Classical and quantum $sl(1|2)$ superalgebras, Casimir operators and quantum chain Hamiltonians}, \textit{J. Math. Phys.} \textbf{36} (1995), 5262--5283.
\bibitem[BH]{BH}
X. ~Bai, N. ~Hu, {Two-parameter quantum groups of exceptional type $E$-series and convex PBW-type basis}, \textit{Algebra Colloq.} \textbf{15} (2008), 619--636.
\bibitem[BW1]{BW1}
G. ~Benkart, S. ~Witherspoon, {A Hopf structure for down-up algebras}, \textit{Math. Z.} \textbf{238} (1989), 523--553.
\bibitem[BW2]{BW2} G. ~Benkart, S. ~Witherspoon, {Two-parameter quantum groups and Drinfel'd doubles}, \textit{Algebr. Represent. Theory} \textbf{7} (2004), 261--286.
\bibitem[BW3]{BW3} G. ~Benkart, S. ~Witherspoon, {Representations of two-parameter quantum groups and Schur-Weyl duality}, in: \textit{Hopf algebras}, Lecture Notes in Pure and Appl. Math. \textbf{237}, pp. 65--92, 2004.
\bibitem[BW4]{BW4} G. ~Benkart, S. ~Witherspoon, {Restricted two-parameter quantum groups}, in: \textit{Representations of finite dimensional algebras and related topics in Lie theory and geometry}, pp. 293--318, Fields Inst. Commun., Providence, RI, 2004.
\bibitem[BGH1]{BGH1} N. ~Bergeron, Y. ~Gao, N. ~Hu, {Drinfel'd doubles and Lusztig's symmetries of two-parameter quantum groups},
\textit{J. Algebra} \textbf{301} (1) (2006), 378--405.
\bibitem[BGH2]{BGH2}
---, {Representations of two-parameter quantum orthogonal and symplectic groups}, in: \textit{Proceedings of the International Conference on Complex Geometry and Related Fields},
pp. 1--21,  AMS/IP Stud. Adv. Math., \textbf{39}, Amer. Math. Soc., Providence, RI, 2007.
\bibitem[Ca]{Ca}
 X. ~Cao, {Some new properties of Dickson polynomials}, \textit{Beijing Daxue Xuebao Ziran Kexue Ban} \textbf{40} (2004), 12--18.
 \bibitem[CH]{CH}
X. ~Chen, N. ~Hu, X. ~Wang, {Convex PBW-type Lyndon basis and  two-parameter quantum groups of type ${F_4}$}, \textit{Acta Math. Sin. (Engl. Ser.)}
(to appear), 2022.
\bibitem[CL]{CL}
 J. ~Chen, J. ~Liu, {$R$-matrix for the two parameter quantum superalgebra  ${U_{r,s}}\left({osp(1|2n)}\right)$}, \textit{J. Math. Phys.} \textbf{57} (2) (2016), Paper No. 021706, 13 pp.
\bibitem[HM]{HM}
 A. ~Hegazi, M. ~Mansour, {Two-parameter quantum deformation of Lie superalgebras}, \textit{Chaos Solitons Fractals} \textbf{12} (2001), 445--452.
 \bibitem[HP1]{HP1}
 N. ~Hu, Y. ~Pei, {Notes on $2$-parameter quantum groups, I}, \textit{Sci. China Ser. A} \textbf{51} (6) (2008), 1101--1110.
 \bibitem[HP2]{HP2}
 ---, {Notes on two-parameter quantum groups, (II)}, \textit{Comm. Algebra} \textbf{40} (9) (2012), 3202--3220.
\bibitem[HSQ]{HSQ}
 N. ~Hu, Q. ~Shi, {The two-parameter quantum group of exceptional type ${G_2}$ and Lusztig's symmetries}, \textit{Pacific J. Math.} \textbf{230} (2) (2007), 327--345.
 \bibitem[HSY]{HSY}
N. ~Hu, Y. ~Shi, {On the centre of two-parameter quantum groups ${U_{r,s}}(\mathfrak{g})$ for type ${B_n}$ with $n$ even}, \textit{J. Geom. Phys.} \textbf{86} (2014), 422--433.
\bibitem[HW]{HW}
 N. ~Hu, X. ~Wang, {Convex PBW-type Lyndon basis and restricted two-parameter quantum groups of type ${G_2}$}, \textit{Pacific J. Math.} \textbf{241} (2) (2009), 243--273.
 \bibitem[Ja]{Ja}
 J.~C. ~Jantzen, {Lectures on Quantum Groups}, \textit{Graduate Studies in Math.} \textbf{6}, Amer. Math. Soc, Providence, RI, 1996.
\bibitem[JL1]{JL1} N. ~Jing, M. ~Liu, {$R$-matrix realization of two-parameter quantum group ${U_{r,s}}(g{l_n})$}, \textit{Commun. Math. Stat.} \textbf{2} (2014), 211--230.
\bibitem[JL2]{JL2} N. ~Jing, M. ~Liu, {On fusion procedure for the two-parameter quantum algebra in type $A$}, \textit{Bull. Inst. Math. Acad. Sin. (N.S.) } \textbf{14} (2019), 15--29.
\bibitem[Kac]{Kac} V.~G. ~Kac, {Representations of classical Lie superalgebras, in Lecture Notes in Math.} \textbf{676}, \textit{Springer-Verlag}, 1978.
\bibitem[Kas]{Kas} C. ~Kassel, {Quantum Groups}, \textit{Graduate Texts in Math.} \textbf{155}, \textit{Springer-Verlag}, 1995.
\bibitem[Ke]{Ke}
 T. ~Kerler, {Darstelungen der Quantengruppen und Anwendungen}, \textit{Diplomarbeit, ETH-Zurich}, 1985.
 \bibitem[Le]{Le}
 A. ~Le\'sniewski, {A remark on the Casimir elements of Lie superalgebras and quantized Lie superalgebras}, \textit{J. Math. Phys.} \textbf{36} (1995),
 1457--1461.
 \bibitem[LMT]{LMT}
R. ~Lidl, G.L. ~Mullen, G. ~Turnwald, {Dickson polynomials}, \textit{Pitman Monographs and Surveys in Pure and Applied Math.}, 1993.
\bibitem[LXZ1]{LXZ1}
L. ~Li, L. ~Xia, Y. ~Zhang, {On the centers of quantum groups of $A_n$-type}, \textit{Sci. China Math.} \textbf{61} (2) (2018), 287--294.
\bibitem[LXZ2]{LXZ2}
---, {On the center of the quantized enveloping algebra of a simple Lie algebra}, arXiv:1607.00802.
\bibitem[Lu]{Lu}
 G.~Lusztig, {Modular representations and quantum groups},  \textit{Contemp. Math.} \textbf{82} (1989), 59--77.
\bibitem[LWY]{LWY}
Y. ~Luo, Y. ~Wang, Y. ~Ye, {On the Harish-Chandra homomorphism for quantum superalgebras}, \textit{Comm. Math. Phys.} \textbf{393} (3) (2022),  1483--1527.
\bibitem[Sh]{Sh}
 G. ~Shi, {Two-parameter quantum supergroups of the Lie superalgebra $osp(1|2n)$} (Chinese), \textit{J. East China Norm. Univ. Natur. Sci. Ed.}, no. 5, (2011), 121--132.
\bibitem[Zh]{Zh}
 H. ~Zhang, {Two-parameter quantum general linear supergroups},
 In: \textit{Quantum theory and symmetries with Lie theory and its applications in physics}, Vol. \textbf{1}, pp. 367--376,
Springer Proc. Math. Stat., \textbf{263}, Springer, Singapore, 2018.


\end{thebibliography}

\end{document}